\documentclass{amsart}
\usepackage{amscd}
\usepackage{amssymb}
\def\Kweb#1{
http:\linebreak[3]//www.\linebreak[3]math.\linebreak[3]uiuc.
\linebreak[3]edu/\linebreak[3]{K-theory/#1/}}

\def\Star{\operatorname{Star}}
\def\chr{\operatorname{char}}
\def\orb{\operatorname{orb}}

\newcommand{\mathdot}{{\mathbf{\scriptscriptstyle\bullet}}}

\def\cF{\mathcal F}

\def\cO{\mathcal O}

\def\coker{\operatorname{coker}}

\def\Hom{\operatorname{Hom}}

\def\Spec{\operatorname{Spec}}

\def\cdh{\text{cdh}}
\def\zar{\text{Zar}}
\def\lra{\longrightarrow}
\def\map#1{{\buildrel #1 \over \lra}}
\def\lmap#1{{\buildrel #1 \over \longleftarrow}}
\def\into{\rightarrowtail}
\def\onto{\twoheadrightarrow}

\def\vecthree#1#2#3{\scriptscriptstyle
    \begin{bmatrix} #1 \\ #2 \\ #3 \end{bmatrix}}

\def\tOmega{\tilde{\Omega}}

\newcommand{\Q}{\mathbb{Q}}

\newcommand{\bH}{\mathbb{H}}

\newcommand{\R}{\mathbb{R}}
\newcommand{\F}{\mathbb{F}}
\newcommand{\bL}{\mathbb{L}}

\newcommand{\C}{\mathbb{C}}
\newcommand{\Z}{\mathbb{Z}}
\newcommand{\N}{\mathbb{N}}

\def \ie{{\it i.e.,\ }}

\newcommand{\fc}{{\mathfrak c}}

\numberwithin{equation}{section}

\theoremstyle{plain}
\newtheorem{thm}[equation]{Theorem}
\newtheorem{cor}[equation]{Corollary}
\newtheorem{lem}[equation]{Lemma}
\newtheorem{prop}[equation]{Proposition}

\theoremstyle{definition}
\newtheorem{defn}[equation]{Definition}
\newtheorem{ex}[equation]{Example}

\theoremstyle{remark}
\newtheorem{rem}[equation]{Remark}
\newtheorem{subremark}{Remark} [equation]

\begin{document}

\bibliographystyle{plain}

\title{The $K$-theory of toric varieties}

\author{G.~ Corti\~nas}
\thanks{Corti\~nas' research was partially supported by FSE and by grants
ANPCyT PICT 03-12330, UBACyT-X294, JCyL VA091A05, and MEC MTM00958.}
\address{Dep. Matem\'atica, FCEyN-UBA\\ Ciudad Universitaria Pab 1\\
1428 Buenos Aires, Argentina\\ and  Dep. \'Algebra\\ Fac. de Ciencias\\
Prado de la Magdalena s/n\\ 47005 Valladolid, Spain.}
\email{gcorti@agt.uva.es}\urladdr{http://mate.dm.uba.ar/\~{}gcorti}

\author{C.~ Haesemeyer}
\address{Dept.\ of Mathematics, University of Illinois, Urbana, IL
61801, USA} \email{chh@math.uiuc.edu}

\author{Mark E.~Walker}
\thanks{Walker's research was supported by NSF grant DMS-0601666.}
\address{Department of Mathematics \\
University of Nebraska -- Lincoln \\
Lincoln \\ NE \\ 68588-0130 \\ U.S.A.}
\email{mwalker5@math.unl.edu}

\author{C.~ Weibel}
\thanks{Weibel's research was
supported by NSA grant MSPF-04G-184 and the Oswald Veblen Fund}
\address{Dept.\ of Mathematics, Rutgers University, New Brunswick,
NJ 08901, USA} \email{weibel@math.rutgers.edu}

\date{\today}

\begin{abstract}
Recent advances in computational techniques for $K$-theory allow us to
describe the $K$-theory of toric varieties in terms of the $K$-theory of
fields and simple cohomological data.
\end{abstract}



\maketitle

\section{Introduction}

In this paper, we revisit the $K$-theory of toric varieties, using the
new perspective afforded by the recent papers \cite{HKH}, \cite{CHSW},
\cite{CHW}.  These papers provide a new technique for
computations of the $K$-theory of a singular algebraic variety $X$ over
a field of characteristic $0$, in terms of the homotopy $K$-theory
of $X$ and cohomological data:
the cyclic homology of $X$ and the $cdh$-cohomology of the
sheaves $\Omega^p$ of K\"ahler differentials.

The homotopy $K$-theory $KH_*(X)$ of an affine toric variety is just the
algebraic $K$-theory of
a Laurent polynomial ring, and is well understood. Even when $X$ is a
non-affine toric variety, $KH_*(X)$ is tractable; we show in Proposition
\ref{Willies} that it is a summand of $K_*(X)$. This allows us to give
a short proof in Proposition \ref{K0} of Gubeladze's classical theorem
(in \cite{Gu88}) that $K_0(X)=\Z$ for affine $X$.

This reduces the problem of understanding $K_*(X)$ to that of
understanding the cyclic homology of $X$ and its $cdh$-cohomology.
Because toric varieties admit resolutions of singularities that are
formed in a purely combinatorial manner, it turns out this is indeed
an accessible problem.

The main goal of this paper is to use these new techniques to give
a streamlined approach to two of Gubeladze's recent results concerning the
$K$-theory of toric varieties: examples of toric varieties with ``huge''
Grothendieck groups \cite{Gu04} and his ``Dilation Theorem'' (verifying
the ``nilpotence conjecture'') \cite{Gu05}.
Our proof of this theorem is considerable shorter than the original.
On the other hand, our approach and Gubeladze's are cousins in the sense
that they have a common ancestor:
Corti\~nas' verification of the KABI conjecture \cite{Cor}.

Since varieties are locally smooth in the $cdh$-topology, it is not
surprising that the $cdh$-fibrant version of cyclic homology is
strongly related to the
$cdh$ cohomology of the sheaf $\Omega^p$ of K\"ahler differentials.
Theorem \ref{Thm1} below shows that, for a toric variety $X$,
the $cdh$ cohomology of $\Omega^p$ is computed by the Zariski cohomology of
Danilov's sheaf of differentials $\tOmega^q_X$.
Since the global sections of $\Omega^p_X$ and $\tOmega^p_X$ can be computed
explicitly for toric varieties, we are able to find easily examples
of toric varieties with huge Grothendieck groups; see Example \ref{huge}.

Gubeladze's Dilation Theorem (stated and proven in Theorem \ref{Thm3} below)
asserts, roughly speaking, that after inverting the action of ``dilations,''
the $K$-theory of a toric variety becomes homotopy invariant.
Our Theorem \ref{Thm2} shows
that, after inverting the action of dilations, the global sections of
$\tOmega^q_X$ agree with the Hochschild homology groups
$HH_q(X)$. By the technique of \cite{CHSW}, this quickly leads to
our new proof of Gubeladze's theorem.

\medskip

{\em Notation.}
Throughout this paper, we will adhere to  the following notation.  Let $N$ be
a free abelian group of rank $n < \infty$ and let $M = N^* = \Hom(N,
\Z)$. Define $N_\R = N \otimes_\Z \R$ and $M_\R = \Hom_\R(N_\R, \R)
\cong M \otimes_\Z \R$. For $m \in M_\R, n \in N_\R$,
let $\langle m,n \rangle$ denote the value of $m$ at $n$. Finally, let
$k$ denote a field of characteristic $0$.

\section{Review of toric varieties}

The material in this section may be found in standard texts, such as
\cite{Fulton} or \cite{Dani}.

A {\em strongly convex rational cone} in $N_\R$ is a subset $\sigma
\subset N_\R$ that is a cone spanned by finitely many vectors in $N$
and that contains no lines. That is,
$\sigma = \R_{\geq 0} v_1 + \cdots + \R_{\geq 0} v_k$
for some $v_1, \dots , v_k \in N \subset
N_\R$ and whenever both $u$ and $-u$ belong to $\sigma$, we must have
$u = 0$.  Given such a cone $\sigma$, let $\sigma^\vee \subset M_\R$
denote the dual cone, defined to consist of those $m \in M_\R$ such
that $\langle m, - \rangle \geq 0$ on $\sigma$.  Note that
$\sigma^\vee \cap M$ is the
abelian monoid (under addition of functions) of linear functions with
integer coefficients on $N_\R$ whose restrictions to $\sigma$ are
nowhere negative.  A {\em face} of $\sigma$ is a subset $\tau$ of the
form
\begin{equation}\label{sigma-m}
 \sigma(m) = \{n \in \sigma \, | \, \langle m,n \rangle = 0 \}
\end{equation}

for some $m \in \sigma^\vee$.
Observe that a face of a strongly convex rational cone
is again a strongly convex rational cone.  We write $\tau \prec
\sigma$ to indicate that $\tau$ is a face of $\sigma$.

Recall that $k$ denotes a field of characteristic zero.

The affine toric $k$-variety associated to a strongly convex rational
cone $\sigma$ is $U_\sigma = \Spec k[\sigma^\vee \cap M]$.
We write elements of the monoid ring
$k[\sigma^\vee \cap M]$ as $k$-linear combinations of the set of formal symbols
$\{ \chi^m \, | \, m \in \sigma^\vee \cap M\}$,
so that multiplication in this ring is given on this $k$-basis by
$\chi^m \cdot \chi^{m'}= \chi^{m+m'}$.

A {\em fan} $\Delta$ in $N_\R$ is a finite collection of strongly convex
rational cones in $N_\R$ such that (1) any face of a cone in $\Delta$
is again in $\Delta$ and (2) the intersection of any two cones in
$\Delta$ is a face of each. If $\tau$ is a face of $\sigma$, then
$U_\tau \to U_\sigma$ is an open immersion, because the evident map
$k[\sigma^\vee \cap M] \to k[\tau^\vee \cap M]$ is given by
inverting a finite number of the $\chi^m$.
It follows that for any fan $\Delta$, we may form a scheme $X(\Delta)$
by patching together the affine schemes $U_\sigma$ corresponding to

cones $\sigma$ along the open
subschemes associated to their intersections.

We call $X(\Delta)$ the {\em toric variety} associated to $\Delta$.

\medskip
{\it Orbits.}
We write $T_N = \Spec k[M]$ for the $n$-dimensional torus associated
to $N$. Observe that $T_N$ acts on each $U_\sigma$ --- equivalently, the
ring $k[\sigma^\vee \cap M]$ is naturally $M$-graded with weight $m$
part being $k \cdot \chi^m$, if $m \in \sigma^\vee$, and $0$ if $m
\notin \sigma^\vee$. Since these actions are compatible,
the torus $T_N$ acts on $X(\Delta)$.

The orbits of this action are tori, and are in 1--1
correspondence with the cones of $\Delta$; thus $X(\Delta)$
is the disjoint union of the orbits $\orb(\tau)$
corresponding to the $\tau \in \Delta$.
To describe the orbit for $\tau$, let $\Z(\tau \cap N)$ denote the subgroup of
$N$ generated by $\tau \cap N$, and let $\overline{N}$ be the free
abelian group $N/\Z(\tau \cap N)$. Then $\orb(\tau) \cong T_{\overline{N}}$.
Note that the orbit corresponding to
the minimal cone $\{0\}$ is the dense open $\orb(0)=U_0$, and is
naturally isomorphic to $T_N$.

We write $V_\Delta(\sigma)$ for the closure of $\orb(\sigma)$ in
$X(\Delta)$. The orbits in $V_\Delta(\sigma)$ are indexed by
the {\it star} of $\sigma$, $\Star_\Delta(\sigma)$,
defined as the set of cones in $\Delta$ containing $\sigma$:
$$
V_\Delta(\sigma) = \coprod_{\sigma \prec \tau}
\orb(\tau).
$$
Each  orbit-closure $V_\Delta(\sigma)$ has the structure of a toric
variety. To see this, let $\overline{N} = N/\Z(\sigma \cap N)$. Then
$\{\overline{\epsilon} \, | \, \sigma \prec \epsilon \}$
forms a fan in $\overline{N}_\R$, and the corresponding toric
variety is $V_\Delta(\sigma)$.
The torus $T_{\overline{N}}$ is a quotient of $T_N$ and the inclusion
$V_\Delta(\sigma) \subset X(\Delta)$ is $T_N$-equivariant, the
action of $T_N$ on $V_\Delta(\sigma)$ being induced by the quotient map
$T_N\onto T_{\overline{N}}$. In the case where $\Delta$
has a single maximal
cone $\epsilon$, so that $\sigma$ is a face of $\epsilon$, we have
$$
V_\Delta(\sigma) =
\Spec k[\epsilon^\vee \cap M \cap  \sigma^\perp],
$$
and the closed immersion
$V_\Delta(\sigma) \hookrightarrow U_\epsilon$ is given by the ring surjection
$$
\Spec k[\epsilon^\vee \cap M] \onto
\Spec k[\epsilon^\vee \cap M \cap  \sigma^\perp]
$$
sending $\chi^m$ to $0$, if $m \notin \sigma^\perp$, and to $\chi^m$,
if $m \in \sigma^\perp$.

It is useful to regard
the open complement of $V_\Delta(\sigma)$ in $X(\Delta)$ as the toric
variety corresponding to the largest sub-fan of $\Delta$ in $N_\R$
that does not contain $\tau$.

Every toric variety is normal, but need not be smooth.  A toric variety
$X(\Delta)$ is smooth if and only if, for every cone $\sigma$ in the
fan $\Delta$, the minimal lattice points along the 1-dimensional faces
(rays) of $\sigma$ form part of a $\Z$-basis of $N$.
In particular, in order for $X(\Delta)$ to be smooth,
the set of rays of each cone must be $\R$-linearly independent
(such a cone is said to be {\em simplicial}).

\medskip
{\it Resolution of Singularities.}
We will need a detailed description of resolutions of singularities
for toric varieties, which we now recall from \cite{Fulton}.
If $v \in N$ is contained in one (or more) of the cones of $\Delta$,
one may subdivide $\Delta$ by the ray $\rho=\R_{\ge0}v$ through $v$
to form a new fan $\Delta'$ in $N_\R$ as follows:
If $\tau \in \Delta$ does not contain $\rho$, then $\tau$
is also a cone of $\Delta'$. For each cone $\tau \in \Delta$
containing $\rho$ and for each face $\nu$ of $\tau$
not containing $\rho$, $\Delta'$ contains the cone spanned by $\rho$ and $\nu$:
$$
\tilde{\nu} := \nu + \R_{\geq 0} \rho.
$$

Finally, $\rho$ itself belongs to $\Delta'$.
Thus if $\sigma \in \Delta$ is the minimal cone of
$\Delta$ containing $\rho$, then $\Delta'$ is the disjoint union of
$\Delta \setminus \Star_\Delta(\sigma)$ and
$\Star_{\Delta'}(\rho)$.

There is a map of toric varieties $X' = X(\Delta') \to X = X(\Delta)$
and it is proper, birational, and equivariant with respect to the action of
the torus $T_N$. Starting with any toric variety $X(\Delta)$, one
can arrive at a desingularization of $X(\Delta)$ by performing a finite
number of subdivisions of this type.

Suppose $\Delta'$ is the fan obtained by subdividing $\Delta$ by
inserting a ray $\rho$, and let $\sigma \in \Delta$ be the minimal cone
in $\Delta$ containing $\rho$. Then the description of the orbit-closures
given above makes it clear that
\begin{equation} \label{E2}
\begin{CD}
\llap{$V'=$} V_{\Delta'}(\rho) @>{i'}>> X(\Delta') \rlap{$=X'$} \\
@V{}VV @VV{\pi}V \\
\llap{$V =$} V_{\Delta}(\sigma) @>{i}>> X(\Delta) \rlap{$=X$}\\
\end{CD}
\end{equation}
is an abstract blow-up square. That is, this a pull-back square in
which the horizontal arrows are closed immersions and the map on open
complements is an isomorphism:
$$
X(\Delta') \setminus V_{\Delta'}(\rho) \map{\cong}
X(\Delta) \setminus V_{\Delta}(\sigma).
$$

As with any abstract blow-up, the maps
$\{X(\Delta') \to X(\Delta), V_{\Delta}(\sigma)\to X(\Delta) \}$
form a covering for the $cdh$-topology.
Recall that
the torus $T_N$ acts on each variety in the above square and each map in this
square is $T_N$-equivariant.

\bigskip\goodbreak
\section{Danilov's sheaves $\tOmega^p$}\label{tOmega}

In this section, we introduce the coherent sheaves $\tOmega^p_X$,
first defined by Danilov \cite[4.2]{Dani}. We will see in the next section
that their Zariski cohomology
groups turn out to give the $cdh$-cohomology groups of $\Omega^p$.

Given a fan $\Delta$, let $\Delta(1)$ denote the collection of rays
in $\Delta$; the 1-skeleton of $\Delta$ is the fan $\Delta(1)\cup\{0\}$
and its toric variety $X^{(1)}$ lies in the smooth locus of $X(\Delta)$.

\begin{defn}
For a toric $k$-variety $X = X(\Delta)$ defined by a fan $\Delta$ in $N_\R$,
we define $\tOmega^p_X$ to be the coherent sheaf on $X$
fitting into the exact sequence
$$
0 \to \tOmega^p_X \to \cO_X \otimes_\Z \wedge^p(M) \map{\delta}
\bigoplus_{\rho \in\Delta(1)} \cO_{V_\Delta(\rho)} \otimes_\Z
                \wedge^{p-1}(M \cap \rho^\perp).
$$

The component of the map $\delta$ indexed by $\rho$
sends $f \otimes (m_1 \wedge \cdots \wedge m_p)$ in
$\cO_X \otimes \wedge^p_\Z(M)$ to
$$
i^*(f) \otimes
\left(\sum_{i} (-1)^i \langle m_i, n_\rho \rangle
m_1 \wedge \cdots \wedge \hat{m_i} \wedge
\cdots \wedge m_p \right)
$$
where $i:V_\Delta(\rho) \hookrightarrow X$ is the canonical closed immersion,
and $n_\rho \in N$ is the minimal lattice point on $\rho$.
By convention, $\tOmega^0_X=\cO_X$.
\end{defn}

On the affine $U_\sigma$, the ring $\cO(U_\sigma)$ is $M$-graded,
so the sections
of $\cO_X\otimes\wedge^pM$ are $M$-graded with $\wedge^pM$ in weight~0;
the weight $m$ summand is $k\cdot\chi^m\otimes\wedge^pM$ if
$m\in\sigma^\vee$.
Since $\delta$ is graded, it follows that each $\tOmega^p_X(U_\sigma)$ is
$M$-graded.

\begin{rem}\label{log-poles}
Sections of $\tOmega^1_X$ may be considered as differential forms on $X$,
with $1\otimes m$ corresponding to the form $d\log(\chi^m)=d\chi^m/\chi^m$.
On a nonsingular cone $\sigma$, we may identify $\cO\otimes\wedge^p M$ with
the locally free sheaf $\Omega^p(\log D)$ of differentials with
logarithmic poles along $D=\cup V(\rho)$. This identifies the map
$\delta$ with the residue map, so we have
$\Omega^p\vert_{U_\sigma} \cong \tOmega^p\vert_{U_\sigma}$.
\end{rem}

As shown by Danilov \cite[4.3]{Dani},
the sheaf $\tOmega^p_X$ is naturally isomorphic to $j_*(\Omega^p_U)$,
where $j: U \hookrightarrow X$
is the immersion of the open subscheme $U$ of smooth points of $X$.
Applying Remark \ref{log-poles} to $X^{(1)}\hookrightarrow U$,
we see that $\tOmega^p_X = j^{(1)}_*(\Omega^p_{X^{(1)}})$ where

$j^{(1)}: X^{(1)} \hookrightarrow X$ is the evident open immersion.

We will need an explicit description of the $M$-grading on $\tOmega^1$,
or rather on the module of sections $\tOmega^1(U_\sigma)$ over
an affine toric variety $U_\sigma$. (See \cite[4.2.3]{Dani}.)
When $m \in \sigma^\vee \cap M$,
its weight $m$ summand is the subspace
$\tOmega^1(U_\sigma)_m= k\cdot\chi^m\otimes(M \cap \sigma(m)^\perp)$
of the weight $m$ summand $k\cdot\chi^m\otimes M$ of
$\cO(U_\sigma)\otimes M$.
Here $\sigma(m)^\perp$ is the orthogonal complement of the face $\sigma(m)$
of $\sigma$ defined in \eqref{sigma-m} by the vanishing of $m$:
For $m \notin \sigma^\vee$, $\tOmega^1(U_\sigma)_m = 0$
because $\cO(U_\sigma)_m=0$.
More generally, we have for $m \in M$ and $p\ge0$
\begin{equation} \label{E1}
\tOmega^p_X(U_\sigma)_m =
\begin{cases} k\cdot \chi^m\otimes
\wedge^p(M \cap \sigma(m)^\perp) & \text{if $m \in \sigma^\vee$} \\
0 & \text{if $m \notin \sigma^\vee$.} \\
\end{cases}
\end{equation}

It is instructive to compare \eqref{E1} to the analogous formula for
$\Omega^p(U_\sigma)$ and $HH_p(U_\sigma)$,
which are graded by the submonoid $\sigma^\vee\cap M$ of $M$.
There is a natural map from the module $\Omega^p_X$ of K\"ahler
differentials to $\tOmega^p_X$. On $U_\sigma$ it is the $M$-graded map
induced by the $M$-graded map $\Omega^p(U_\sigma)\to
\cO(U_\sigma)\otimes\wedge^p(M)$ defined by:
\begin{equation}\label{E3}
\chi^{m_0}\, d\chi^{m_1}\wedge\cdots\wedge d\chi^{m_p} \mapsto
(1/p!)\chi^{m} \otimes \bigl( m_1\wedge\cdots\wedge m_p \bigr),
\qquad m=\sum m_i.
\end{equation}

Recall that the orbit-closure $V(\tau)$ for the face $\tau$ is
$\Spec(k[\sigma^\vee\cap M\cap\tau^\perp])$.

\goodbreak
\begin{lem}\label{HH-m}
For each $m\in\sigma^\vee\cap M$, let $V=V(\sigma(m))$ denote the
orbit-closure for the face $\sigma(m)$ of $\sigma$.
Then the closed immersion $V\subset U_\sigma$
induces an isomorphism
$HH_*(U_\sigma)_m \cong HH_*(V)_m$. In particular, for all $p$:
\[
\Omega^p_X(U_\sigma)_m = \Omega^p(V)_m
\]
\end{lem}

\begin{proof}
For convenience, let us set $A=\sigma^\vee\cap M$ and
$B=A\cap\sigma(m)^\perp$, so that $U_\sigma=\Spec(k[A])$ and
$V(\sigma(m))=\Spec(k[B])$. The immersion $V\subset U_\sigma$
corresponds to a surjection $k[A]\to k[B]$, which is split by the evident
inclusion $\iota:k[B]\to k[A]$. Hence $HH_*(k[B])$ is a summand of
$HH_*(k[A])$, and it suffices to show that $\iota$ induces a surjection
on the weight $m$ summand of the complex for Hochschild homology.

Now the degree $p$ part of the Hochschild complex for $k[A]$ is
$k[A]^{\otimes p+1}$, so the weight $m$ summand has a basis consisting of
the $\chi^{u_0}\otimes\chi^{u_1}\cdots\otimes\chi^{u_p}$ where $u_i\in A$
and $\sum u_i=m$. If $n\in\sigma(m)$, then $\langle u_i,n\rangle \ge0$
and $\sum_i \langle u_i,n\rangle =\langle m,n\rangle =0$. This forces
each $\langle u_i,n\rangle =0$, \ie  $u_i\in B$.
Hence $k[B]^{\otimes p+1}_m = k[A]^{\otimes p+1}_m$, as claimed.
\end{proof}

\begin{lem}\label{orbitBU}
Every orbit blow-up square \eqref{E2} determines
a distinguished triangle on $X_\zar$ of the form
$$
\tOmega^p_X \to \R\pi_*\tOmega^p_{X'}\oplus i_*\tOmega^p_V \to
\R\pi_*i'_*\tOmega^p_{V'} \to \tOmega^p_X[1],
$$
and hence a long exact sequence of Zariski cohomology groups:
$$
\cdots \to H^q(X, \tOmega^p)
\to H^q(X', \tOmega^p)  \oplus H^q(V, \tOmega^p) \to
H^q(V', \tOmega^p)
\to H^{q+1}(X, \tOmega^p)  \to \cdots.
$$
\end{lem}

\begin{proof}
We have short exact sequences of coherent sheaves
$$
0 \to \tOmega^p_{(X,V)} \to \tOmega^p_X \to i_* \tOmega^p_V \to 0
$$
on $X$, and
$
0 \to \tOmega^p_{(X',V')} \to \tOmega^p_{X'} \to i_* \tOmega^p_{V'} \to 0
$
on $X'$.
Applying $\R\pi_*$ to the latter yields a morphism of distinguished triangles
$$
\begin{CD}\minCDarrowwidth=12pt
\tOmega^p_{(X,V)} @>>> \tOmega^p_X @>>> i_* \tOmega^p_V\vspace{-3pt} \\
@VVV @VVV @VVV \vspace{-3pt} \\
\R\pi_* \tOmega^p_{(X',V')} @>>> \R\pi_* \tOmega^p_{X'} @>>>
\R \pi_* i_* \tOmega^p_{V'}
\end{CD}
$$
Danilov proved in \cite[Prop 1.8]{Dan1479} that
the left vertical map is a quasi-isomorphism, \ie that
$\R^j\pi_*\tOmega^p_{(X',V')}=0$ for $j>0$, and
$\tOmega^p_{(X,V)}\map{\simeq}\pi_* \tOmega^p_{(X',V')}$.
The distinguished triangle follows from this in a standard way.
\end{proof}

\begin{rem}\label{ratsing}
Danilov \cite[8.5.1]{Dani} proved that if $\pi:X'\to X$ is a morphism of
toric varieties resulting from a subdivision of the fan, then
$\cO_X\map{\simeq}\R\pi_*\cO_{X'}$, \ie
$\pi_*\cO_{X'}=\cO_X$ and $R^i\pi_*\cO_{X'}=0$ for $i>0$.
This proves that toric varieties have (at most) rational singularities.
\end{rem}

\goodbreak
\section{The $cdh$-cohomology of $\Omega^p$ for toric varieties}

In this short section, we prove Theorem \ref{Thm1}, that
Danilov's sheaves compute the $cdh$-cohomology groups
$H^*_\cdh(X, \Omega^p)$ for toric varieties.

\begin{thm} \label{Thm1}
Let $X$ be an arbitrary toric $k$-variety. There is an isomorphism
$$
H^*_\zar(X, \tOmega^p_X) \cong H^*_\cdh(X, \Omega^p)
$$
for all $p$, natural for morphisms of toric varieties and for
the closed embedding of an orbit-closure of $X$ into $X$.
\end{thm}

\begin{ex}
The case $*=0$ of Theorem \ref{Thm1} is that
$\tOmega^p(X)\cong H^0_\cdh(X,\Omega^p)$. This is equivalent to
Danilov's calculation \cite[1.5]{Dan1479} that in \eqref{E2},
$\tOmega^p_X \map{\simeq}\pi_*\tOmega^p_{X'}$ for all $p$.
\end{ex}
\goodbreak

For the proof, we recall that $H^*_\cdh(X,\Omega^p)$ is just
the Zariski hypercohomology of the complex $\R a_*a^*\Omega^p|_X$,
where $a: (Sch/k)_\cdh \to (Sch/k)_\zar$ is the morphism of sites
and $|_X$ denotes the restriction from the big Zariski site
$(Sch/k)_\zar$ to $X_\zar$.

Recall that we can resolve the singularities of a toric variety via
equivariant blow-up squares of the form \eqref{E2}.
Iterating the orbit blow-up operations described in \eqref{E2},
as in \cite[6.2.5]{HodgeIII}
we can find a smooth toric $cdh$-hypercover $\pi:Y_\mathdot\to X$.
The following Mayer-Vietoris lemma is an immediate consequence of
\cite[12.1]{SVBK}.

\begin{lem}\label{RaRpi} For every $cdh$ sheaf $\cF$,
$\R a_*\cF|_X \cong \R\pi_*(\R a_*\cF|_{Y_\mathdot})$.
\end{lem}

\begin{proof}[Proof of Theorem \ref{Thm1}]
As in \cite[5.2.6]{HodgeIII}, Lemma \ref{orbitBU} implies that
the maps $\tOmega^p_X\to\R\pi_*\tOmega^p_{Y_\mathdot}$
are quasi-isomorphisms. By Remark \ref{log-poles}, the maps
$\Omega^p_{Y_\mathdot}\to\tOmega^p_{Y_\mathdot}$ are isomorphisms.
Hence we have quasi-isomorphisms of complexes of Zariski sheaves on $X$:
\[
\R\pi_*\Omega^p_{Y_\mathdot} \map{\simeq} \R\pi_*\tOmega^p_{Y_\mathdot}
\lmap{\simeq} \tOmega^p_X.
\]

\noindent
Now by \cite[2.5]{CHW}, we have
$\Omega^p_{Y_n}\cong\R a_*a^*\Omega^p|_{Y_n}$.
Applying Lemma \ref{RaRpi} to $\cF=a^*\Omega^p$ yields:
\[
\R a_*a^*\Omega^p|_X ~\map{\simeq}~ \R\pi_*(\R a_*a^*\Omega^p|_{Y_\mathdot})
\cong \R\pi_*\Omega^p_{Y_\mathdot}.
\]
Applying $H^*_\zar(X,-)$ yields
$H^*_\cdh(X,\Omega^p) \map{\simeq} H^*_\zar(Y_\mathdot, \Omega^p)
\cong H^*_\zar(X, \tOmega^p)$, an isomorphism
which is natural in the pair $Y_\mathdot\to X$. As any two smooth toric
hypercovers have a common refinement, the isomorphism
$\tOmega^p_X \simeq \R a_*a^*\Omega^p|_X$ in the derived category
is independent of $Y_\mathdot$. The asserted naturality follows.
\end{proof}

Now recall that every variety is locally smooth for the $cdh$ topology.
Hence the Hochschild-Kostant-Rosenberg theorem
implies that the Hochschild homology
sheaf $HH_n$ has $a^*HH_n\cong a^*\Omega^n$.
We write $\bH_\cdh(X,HH)$ for $\R a_*a^*$ applied to the Hochschild complex,
and $\bH_\cdh(X,HH^{(t)})$ for its summand in Hodge weight $t$. We write
the Zariski hypercohomology of these complexes as $\bH^*_\cdh(X,HH)$ and
$\bH^*_\cdh(X,HH^{(t)})$, respectively. By \cite[2.2]{CHW},
$\bH_\cdh(X,HH^{(t)}) \cong \R a_*a^*\Omega^t[t]$. Hence
Theorem \ref{Thm1} translates into the following language:

\begin{cor}\label{HH/k}
For every toric variety $X$,
$\bH^{n}_\cdh(X,HH^{(t)})\cong H_\zar^{t+n}(X,\tOmega_X^t)$, and
$$
\bH^{n}_\cdh(X,HH) \cong
\bigoplus\nolimits_{t\ge0}H_\zar^{t+n}(X,\tOmega_X^t).
$$
\end{cor}

The Hochschild homology in \ref{HH/k} is taken over any field $k$ of
characteristic zero. Since every toric variety $X=X_k$ over $k$ is
obtained by base-change from a toric variety $X_{\Q}$
over the ground field $\Q$, flat base-change yields
$\Omega^*_{X/k}\cong\Omega^*_{X_\Q/\Q}\otimes_{\Q} k$, and
the K\"unneth formula yields
$\Omega^*_{X/\Q}=\Omega^*_{X_\Q/k}\otimes_\Q\Omega^*_{k/\Q}
= \Omega^*_{X/k}\otimes_k\Omega^*_{k/\Q}$.
Similar formulas hold for $HH_*(X/\Q)$ and hence for
$\bH^*_\cdh(X,HH(-/\Q))$.

We define $\tOmega^t_{X/\Q}$ to be $j_*\Omega^t_{X/\Q}$. The above
remarks imply that $\tOmega^t_{X}\cong \tOmega^t_{X_\Q/\Q}\otimes_{\Q} k$,
and that there is also a K\"unneth formula
$\tOmega^*_{X/\Q}\cong\tOmega^*_{X}\otimes_k\Omega^*_{k/\Q}$.

Hence we have have the following variant of the previous corollary.

\begin{cor}\label{HH/Q}
For every toric $k$-variety $X$,
\[
\bH^{n}_\cdh(X,HH^{(t)}(-/\Q))\cong H_\zar^{t+n}(X,\tOmega_{X/\Q}^t)
\cong
\bigoplus_{i+j=t}H_\zar^{t+n}(X,\tOmega_{X}^i)\otimes_k\Omega^j_{k/\Q},
\] and
$$
\bH^{n}_\cdh(X,HH(-/\Q)) \cong
\bigoplus_{t\ge0}H_\zar^{t+n}(X,\tOmega^t_{X/\Q}).
$$
\end{cor}

\medskip\goodbreak
\section{$K$-theory and cyclic homology of toric varieties}

Recall from section \ref{tOmega} that $\tOmega^p_X$ has both a combinatorial
definition, and an interpretation as $j_*\Omega^p_U$ where
$j:U\hookrightarrow X$ is the inclusion of the smooth locus.
In this section, we study the exterior differentiation map
$d: \tOmega^p_X \to \tOmega^{p+1}_X$ which arises as the pushforward
of the de Rham differential $d:\Omega^p_U \to \Omega^{p+1}_U$.
The following combinatorial description of this map is useful.

\begin{lem} (\cite[4.4]{Dani})
The map $d: \tOmega^p_X \to \tOmega^{p+1}_X$ induced by exterior
differentiation $d:\Omega^p_U \to \Omega^{p+1}_U$ is the
$M$-graded map which in weight $m$ is
$k\chi^m\otimes(m_1\wedge\cdots)\mapsto
k\chi^m\otimes(m\wedge m_1\wedge\cdots)$. That is, it is induced by:
\[
(\cO_X(U_\sigma)_m \otimes_\Z \wedge^pM) \cong \wedge^pM ~\map{m\wedge-}~
\wedge^{p+1}M \cong (\cO_X(U_\sigma)_m \otimes_\Z \wedge^{p+1}M).
\]
\end{lem}

Pushing forward the de Rham complex $\Omega^*_U$, we see that the
$\tOmega^p_X$'s fit together to form a ``log de Rham'' complex $\tOmega^*_X$
on $X$. There is a natural map $\Omega^*_X \to \tOmega^*_X$ of
complexes, which is an isomorphism on the smooth locus of $X$.
Similarly,
pushing forward the de Rham complex $\Omega^*_{U/\Q}$ from the smooth
locus to $X$, we obtain a log de Rham complex $\tOmega^*_{X/\Q}$.

As in \cite{CHSW} and \cite{CHW},
$\bH_\cdh(X,HC)$ denotes $\R a_*a^*$ applied to the cyclic homology
cochain complex, and $\bH_\cdh(X,HC^{(t)})$ is its summand in Hodge
weight $t$. The Zariski hypercohomology of these complexes is written as
$\bH^*_\cdh(X,HC)$ and $\bH^*_\cdh(X,HC^{(t)})$, respectively,
and is called the {\it $cdh$-fibrant cyclic homology} of $X$.

By \cite[2.2]{CHW},
$\bH_\cdh(X,HC^{(t)}) \cong \R a_*a^*\Omega^{\le t}[2t]$,
where $\Omega^{\le t}$ denotes the brutal truncation of the de Rham complex.
Similarly, we write $\tOmega_X^{\leq t}$ for the brutal truncation of
the Danilov complex $\tOmega^*_X$. By Theorem \ref{Thm1},
$\bH_\cdh(X,HC^{(t)})\cong \tOmega_X^{\leq t}[2t]$.

As with Hochschild homology, the cyclic homology in the above paragraph
is taken over $k$. As in the previous section, we may also consider cyclic
homology taken over the ground field $\Q$, and we also have
$\bH_\cdh(X,HC^{(t)}(-/\Q)) \cong \R a_*a^*\Omega^{\le t}_{/\Q}[2t]$,
again by \cite[2.2]{CHW}.

Again by Theorem \ref{Thm1},
we have an isomorphism in the derived category:
$$\R a_*a^*\Omega^{\le t}_{/\Q} \simeq \tOmega^{\le t}_{X/\Q}.$$
Concatenating these identifications, we have:

\begin{prop} \label{cdhHC}
If $X$ is a toric $k$-variety, %
the $cdh$-fibrant cyclic %
homology is given by the formula:
\[
\bH^{-n}_\cdh(X,HC) \cong
\bigoplus\nolimits_{t\ge0} H_\zar^{2t-n}(X,\tOmega_X^{\le t}).
\]
and
$$
\bH_\cdh^{-n}(X, HC(-/\Q)) \cong
\bigoplus\nolimits_{t \geq 0} H_\zar^{2t-n}(X,\tOmega^{\le t}_{X/\Q}).
$$
\end{prop}

\begin{ex}\label{HC0-cdh}
The case $t=0$ of \ref{cdhHC} yields the formula
$$HC_n^{(0)}(X)= H_\zar^{-n}(X,\cO)\map{\simeq}
H_\cdh^{-n}(X,\cO) = \bH_\cdh^{-n}(X,HC^{(0)}).
$$
This illustrates the interconnections between
the case $p=0$ of Theorem \ref{Thm1}, Danilov's calculation in
Remark \ref{ratsing}, and the convention that $\tOmega^0_X=\cO_X$.
\end{ex}

These calculations tell us about the algebraic $K$-theory of toric varieties,
via the following translation of \cite[1.6]{CHW}
into the present language.

\begin{defn}
Let $\cF_{HC}[1]$ denote the mapping cone complex of
$HC(-/\Q)\to\R a_*a^*HC(-/\Q)$; the indexing we use is such that there is a
long exact sequence:
\[ \cdots \to
H^{-n}(X,\cF_{HC})\to HC_n(X/\Q) \to \bH_\cdh^{-n}(X, HC(-/\Q)) \to\cdots.
\]
\end{defn}

\begin{thm}\label{FKFHC} (\cite[1.6]{CHW})
For every $X$ in $Sch/k$, there is a long exact sequence
$$\cdots\to KH_{n+1}(X) \to
H^{-n}_\zar(X,\cF_{HC}[1])
\to K_n(X)\to KH_n(X) \to\cdots.
$$
\end{thm}

For toric varieties, the sequence \eqref{FKFHC} splits:

\begin{prop}\label{Willies}
For every toric variety $X$,
$K_*(X)\to KH_*(X)$ is a split surjection. Hence
$$
K_n(X)\cong KH_n(X) \oplus H^{-n}_\zar(X,\cF_{HC}[1]).
$$
\end{prop}

\begin{proof}
For each affine cone $\sigma$, $M(\sigma) := M\cap\sigma^\perp$ is a
free abelian monoid, so $T_\sigma=\Spec(k[M(\sigma)])$ is a torus.
We first claim that the inclusion
$i_\sigma: k[M(\sigma)]\into k[M\cap\sigma^\vee]$, or surjection
$U_\sigma\onto T_\sigma$, induces an isomorphism on $KH$-theory, \ie

\addtocounter{equation}{-1}
\begin{subequations}
\begin{equation}\label{eq:KHT}
K(T_\sigma) \map{\simeq} KH(T_\sigma)\map{\simeq} KH(U_\sigma).
\end{equation}

Since \eqref{eq:KHT} factors
$K(T_\sigma)\to K(U_\sigma)\to KH(U_\sigma)$,
this proves the lemma for $U_\sigma$.

Because $T_\sigma$ is regular, the first map is an isomorphism.
For a suitable rational $n\in\sigma$, evaluation at $n$ is a monoid
map from $M\cap\sigma^\vee$ to $\N$ with kernel $M(\sigma)$.

This gives $k[M\cap\sigma^\vee]$ the structure of an $\N$-graded algebra
with $k[M(\sigma)]$ in degree zero. Hence $i_\sigma$ induces an isomorphism
$KH(k[M(\sigma)])\cong KH(k[M\cap\sigma^\vee])$, as claimed.

If $\tau$ is a face of $\sigma$, we have a commutative diagram
\begin{equation*}\begin{CD}
k[M(\sigma)] @>>>k[M\cap\sigma^\vee] \\ @V\text{into}VV @VV\text{into}V \\
 k[M(\tau)] @>>> k[M\cap\tau^\vee].
\end{CD}\end{equation*}
Thus the isomorphism in \eqref{eq:KHT} is natural in $\sigma$, for
$\sigma$ a face of a fan $\Delta$,
and so is the splitting of $K(U_\sigma)\to KH(U_\sigma)$.
Since $K(X)$ is the homotopy limit over $\Delta$ of the $K(U_\sigma)$,
and similarly for $KH(X)$, the homotopy limit of the splittings
provides a splitting of the map $K(X)\to KH(X)$.
\end{subequations}
\end{proof}

\begin{subremark}
The proof amounts to the observation that there is an algebraic homotopy
from $U_\sigma$ onto its smallest orbit $orb(\sigma)$, and that this
homotopy is natural with respect to face inclusions.
\end{subremark}

The sequence \eqref{FKFHC} is compatible with the decomposition
arising from the Adams operations because the Chern character is,
by \cite{chwinf}.
Thus $K^{(i)}_*(X)$ and $KH^{(i)}_*(X)$
fit into a long exact sequence with $\cF_{HC}^{(i-1)}$.
For example, it is immediate from Example \ref{HC0-cdh} that
$\cF_{HC}^{(0)}(X)$ is acyclic, proving that
$K^{(1)}_*(X)\cong KH^{(1)}_*(X)$ for toric varieties. The case $*=0$,
which is a well known assertion about the Picard group of normal varieties,
has the following extension:

\begin{prop}\label{K0}
If $X = U_\sigma$ is an affine toric $k$-variety,
then $K_0(X) = \Z$.
\end{prop}

\begin{proof}
Note that the coordinate ring of $U_\sigma$ is graded, so $KH_0(X)=\Z$.

By
\ref{FKFHC}, we need to show that $\bH^0(X,\cF_{HC})=0$.
Since $HC_{-1}(X)=0$, we are reduced to proving that the map
$$
HC_0(X) \to \bH_\cdh^0(X, HC)
$$
is onto. By \ref{cdhHC}, the target of this map is
$
\bigoplus_{t\ge0} H_\zar^{2t}(X, \tOmega^{\le t}_{/\Q}).
$
Since $X$ is affine, we have
$H_\zar^{2t}(X, \tOmega^{\le t}_{/\Q}) = 0$ for all $t >0$.
Finally, when $t=0$ we have
\[
H_\zar^0(X, \tOmega^{\leq 0}_{/\Q}) = H_\zar^0(X, \cO_X) = HC_0(X).
\qedhere \]
\end{proof}

\begin{subremark} 
A much better version of this Corollary was proven years ago by
  Gubeladze \cite{Gu88}: For a PID $R$, every finitely projective module
  over $R[A]$, where $A$ is a semi-normal, abelian, cancellative
  monoid without non-trivial units, is free. This was extended to the
  case where $R$ is regular by Swan \cite{Swan}.
\end{subremark}
\smallskip\goodbreak

Of course, the dictionary coming from \cite{CHW} via \ref{FKFHC}
also allows us to say something about the
higher $K$-theory of toric varieties. Let $K_n^{(i)}(X)$
denote the weight $i$ part of $K_n(X)\otimes\Q$
with respect to the Adams operations, \ie the eigenspace where
$\psi^k=k^i$ for all $k$. We adopt the parallel notation
$KH^{(i)}_n(X)$ for the weight $i$ part of $KH_n(X)$.

The absolute cotangent sheaf $\bL_X$ of $X/\Q$ has
$\bL_X^{\ge0}=\Omega^1_{X/\Q}$ and $H^{1-n}(X,\bL_X)=HH_{n}^{(1)}(X/\Q)$;
see \cite[8.8.9]{WHomo}. There is a natural map
$\bL_X\to\Omega^1_{X/\Q}\to\tOmega^1_{X/\Q}$.

\begin{cor} \label{cor3}
For any toric $k$-variety $X$, we have a distinguished triangle
\[
\cF^{(1)}_{HC} \to \bL_X \to \tOmega^1_{X/\Q} \to \cF^{(1)}_{HC}[1],
\]
and hence an isomorphism
$K_q^{(2)}(X) \cong KH_q^{(2)}(X) \oplus
H_\zar^{2-q}(X, \bL_X \to \tOmega^1_{X/\Q}).$
\end{cor}

\begin{proof}
The Zariski sheaf $HC^{(1)}$ is the mapping cone of $\cO\to\bL_X$;
see \cite[9.8.18]{WHomo}.
Since $\R a_*(a^*\cO)|X=\cO_X$ by Remark \ref{ratsing}, and
$\bH_\cdh(X,HC^{(1)})\simeq(\cO\to\tOmega^1_X)[2]$ by \ref{cdhHC},
it follows that the mapping cone $\cF_{HC}^{(1)}$ of
$HC^{(1)}\to\bH_\cdh(X,HC^{(1)})$ is also the mapping cone of
$\bL_X\to\tOmega^1_X$. This proves the first assertion; the second
assertion follows from this, Proposition \ref{Willies} and
\cite[1.6]{CHW}.
\end{proof}

The techniques of \cite{CHW} allow us to find examples of toric
varieties with ``huge'' $K_0$ and $K_1$ groups, in the spirit of
\cite{W-87}, \cite{Gu95} and \cite{Gu04}.
Our toric varieties will have quotient singularities
because all the cones will be simplices; see \cite{Fulton}.

\begin{ex}\label{hugeK1}
Let $N = \Z^3$, and let us to agree to write elements of $N$ as column
vectors and elements of $M \cong \Z^3$ as row vectors.  Define
$\tau$ to be the cone in the $xy$-plane of $N_\R = \R^3$ spanned by
the vectors $e_1=\vecthree100$ and $e_1+2e_2=\vecthree120$. Then
$U_\tau$ is a singular, affine toric $k$-variety.

In fact, $U_\tau = \Spec \left(k[X,Y,Z]/(YZ - X^2)[T,T^{-1}]\right)$,
where $X = \chi^{(1,0,0)}$, $Y = \chi^{(0,1,0)}$, $Z = \chi^{(2,-1,0)}$
and $T^{\pm1}= \chi^{(0,0,\pm1)}$\!.
This is because $\tau^\vee \cap M$ is generated by the vectors
$(1,0,0), (0,1,0), (2,-1,0)$ and $(0,0,\pm1)$.

Let $m \in M$ be the vector $(1,0,0)$. Its face is
$\tau(m) = \{0\}$, so $\tau(m)^\perp=M$. We see from \eqref{E1} that
$\tOmega^1(U_\tau)_m =
  k\cdot X\otimes M \cong k^3$. The forms $dX$, $XdY/Y$ and $XdT/T$
form a basis.
On the other hand, $\Omega^1(U_\tau)_m$ is the
  $k$-vector space spanned by $\chi^u d(\chi^v)$ with $u,v \in
  \tau^\vee \cap M$ satisfying $u+v = m$.
It is easy to see that the only $u,v\in\tau^\vee \cap M$ satisfying
$u+v =(1,0,0)$ are when $u,v$ is $\{(0,0,-j),(1,0,j)\}$.
Thus $\Omega^1(U_\tau)_m$ is the 2-dimensional vector space spanned by
$dX$ and $XdT/T$.
It follows that $\Omega^1(U_\tau)\to\tOmega^1(U_\tau)$ is not onto
in weight $m$.

Similar reasoning shows that for $m=(1,0,c)$ we also have
$\tOmega^1(U_\tau)_m\cong k^3$ on $T^c\, dX$, $T^cX\, dY/Y$ and
$T^{c-1}X\, dT$, and that $\tOmega^1(U_\tau)_m=\Omega^1(U_\tau)$ for
all other $m$. (It is useful to use the fact that $\Omega^1(U_\tau)$ is a
submodule of $\tOmega^1(U_\tau)$ by \cite{Vasc}.)
Thus $\tOmega^1(U_\tau)/\Omega^1(U_\tau)\cong k[T,T^{-1}]$.
By the K\"unneth formula,
\[
\coker\bigl\{ \Omega^1(U_\tau/\Q)\to\tOmega^1(U_\tau/\Q)\bigr\} \cong
\tOmega^1(U_\tau)/\Omega^1(U_\tau).
\]

As in Proposition \ref{Willies}, it is easy to see that
$KH_*(U_\tau)\cong K_*(k[T,T^{-1}])$. Hence \ref{cor3} implies that
$K_1^{(2)}(U_\tau)$ is isomorphic to a nonzero $k$-vector space:
\[
K_1^{(2)}(U_\tau)\cong H^1_\zar(U_\tau,\Omega^1\to\tOmega^1) \cong
\tOmega^1(U_\tau)/\Omega^1(U_\tau) \cong k[T,T^{-1}].
\]
\end{ex}

\begin{ex}\label{huge}
We now extend the $\tau$ of Example \ref{hugeK1}
to form a fan $\Delta$ consisting of two
$3$-dimensional cones $\sigma_1$, $\sigma_2$ (together with all of
their faces) such that $\sigma_1 \cap \sigma_2 = \tau$. Specifically,
let $\sigma_1$ and $\sigma_2$ be spanned by the two edges of
$\tau$ together with
\[
v_1={\vecthree{-1}{\phantom{-}0}{+1}} \quad\text{and}\quad
v_2={\vecthree{-1}{\phantom{-}0}{-1}},
\]
respectively. Let $X = X(\Delta)$,
so   $X = U_{\sigma_1}\cup U_{\sigma_2}$ and
$U_\tau = U_{\sigma_1}\cap U_{\sigma_2}$.
It follows from \ref{Willies} that $KH_0(X)=\Z\oplus\Z$ and that
$$
K_0(X)\cong\Z^2\oplus H^1_\zar(X,\cF_{HC}).
$$
We will show that the right-hand term is nonzero; since it is a
$k$-vector space, it will follow that $K_0(X)$ contains
the additive group underlying a non-zero $k$-vector space.
Taking $k$ to be uncountable, for example $k=\C$, we see $K_0(X)$ is
uncountable.

Because the singular locus of $X$ is 1-dimensional,
$H^n(X,\bL_X)=H^n(X,\Omega^1_X)$ for $n>0$. By Corollary \ref{cor3},
$$
K_0^{(2)}(X)= H^1_\zar(X,\cF_{HC})=H^2_\zar(X,\Omega^1\to\tOmega^1).
$$
 From the Mayer-Vietoris sequence for the given cover of $X$, and
Proposition \ref{K0}, we see that there is an exact sequence
\quad
$$
\tOmega^1(U_{\sigma_1})/\Omega^1(U_{\sigma_1}) \oplus
\tOmega^1(U_{\sigma_2})/\Omega^1(U_{\sigma_2}) \to
\tOmega^1(U_{\tau})/\Omega^1(U_{\tau}) \to
K_0^{(2)}(X) \to 0.
$$
By Example \ref{hugeK1},
$\tOmega^1(U_{\tau})/\Omega^1(U_{\tau})$ is zero except in weights
$m=(1,0,c)$, $c\in\Z$, where it is spanned by the forms $T^{c}XdY/Y$.
For such $m$, $\tau(m)=\{0\}$.
If $c>0$ then $m\in\sigma_1^\vee$ and the element
$\chi^m dY/Y\in\tOmega^1(U_{\sigma_1})$ maps to $T^cX\,dY/Y\in\tOmega(U_\tau)$.
If $c<0$ then $m\in\sigma_2^\vee$ and  the element
$\chi^m dY/Y\in\tOmega^1(U_{\sigma_1})$ maps to $T^cX\,dY/Y\in\tOmega(U_\tau)$.

We are left with the form $X\,dY/Y$ in weight $m=(1,0,0)$.
Since $m \notin\sigma_i^\vee$ for $i=1,2$, we have
$\tOmega^1(U_{\sigma_1})_m=\tOmega^1(U_{\sigma_2})_m=0$.
This proves that
\[
K_0^{(2)}(X) \cong \tOmega^1(U_{\tau})/\Omega^1(U_{\tau})_{(1,0,0)} \cong k.
\]
\end{ex}

\goodbreak

As in Gubeladze's example of toric varieties with ``huge''
Grothendieck groups in \cite{Gu04}, we can further extend $\Delta$ to obtain
a complete fan consisting of simplicial cones
$\overline{\Delta}$, so that $\overline{X}
= X(\overline{\Delta})$ is a projective closure of $X$ and such that
$Y = X(\overline{\Delta} - \Delta)$ is smooth. Since $Y$ and $X$ form
an open cover of $\overline{X}$, we see that $K_0(\overline{X})$ also
contains the additive group underlying a non-zero $k$-vector space.

\section{Gubeladze's Dilation Theorem}

The main goal of this section is to give a new proof of Gubeladze's
Dilation Theorem \cite{Gu05} for the $K$-theory of monoid rings, which
we obtain in \ref{cor:DT} as a corollary of a version of this
result valid for all toric varieties (Theorem \ref{Thm3}).

For a toric variety $X = X(\Delta)$ with $\Delta$ a fan in $N_\R$ and
integer $c \in \N$, define $\theta_c: X(\Delta) \to X(\Delta)$ to be
the endomorphism of toric varieties induced by the endomorphism of the
lattice $N$ given by multiplication by $c$. If $\sigma \subset N_\R$
is a cone, the map $\theta_c: U_\sigma \to U_\sigma$ of affine toric
$k$-varieties is induced by the ring endomorphism of $k[\sigma^\vee \cap
  M]$ that sends $\chi^m$ to $\chi^{cm}$. That is, this is the map
that raises all monomials to the $c$-th power. Observe that if $k =
\F_p$ and $c = p$, this is precisely the Frobenius endomorphism, and
it useful to think of $\theta_c$ as a generalization of Frobenius that
exists in the category of toric varieties.

Fix a sequence $\fc = (c_1, c_2, \dots)$ of integers with $c_i\ge2$
for all $i$. If $F$ is a contravariant functor from toric varieties to
abelian groups, we define $F^\fc$ by
$$
F(X)^\fc = \varinjlim \left(
F(X) \map{\theta_{c_1}^*} F(X) \map{\theta_{c_2}^*} \cdots \right).
$$
Gubeladze's Dilation Theorem asserts that the natural map $K_*(X)
\to KH_*(X)$ induces an isomorphism $K_*(X)^\fc \to KH_*(X)^\fc$ for
any toric variety $X$.
Our proof of this theorem involves computing $HH_q(X)^\fc$ where
$HH_*$ denotes Hochschild homology.

Fix a cone $\sigma$. As in the proof of Lemma \ref{HH-m},
the chain complex defining the Hochschild homology
of $k[\sigma^\vee \cap M]$ is $\sigma^\vee\cap M$-graded with the weight of
$\chi^{m_0} \otimes \cdots\otimes \chi^{m_p}$ defined to be
$m_0+\cdots+m_p$, and the Hochschild homology groups of $U_\sigma$ are
$\sigma^\vee\cap M$-graded $k[\sigma^\vee\cap M]$-modules. A fortiori, they
are $M$-graded, with zero in weight $m$ if $m \notin \sigma^\vee$.
Since $\theta_c(\chi^{m_0}\otimes\cdots)=\chi^{cm_0}\otimes\cdots$,
$\theta_c$ sends the weight $m$ summand to the weight $cm$ summand.

The Hochschild homology of a non-affine variety is defined by
taking Zariski hypercohomology of the
sheafification of the complex defined just as in the definition of
$HH_*(R)$, but with $\cO_X \otimes_k \cdots \otimes_k \cO_X$ in place
of $R \otimes_k \cdots \otimes_k R$ (see \cite[4.1]{WG}).

For a toric variety $X = X(\Delta)$, we may compute $HH_*(X)$
as follows:
Let $\sigma_1,\dots,\sigma_m$ denote the maximal cones in the fan $\Delta$.
For each $1 \leq i_0 \leq \cdots \leq i_p \leq m$, we may form the complex
defining the Hochschild homology of the affine toric variety
$U_{\sigma_{i_0} \cap \cdots \cap \sigma_{i_p}}$. We then assemble
these into a bicomplex in the usual {\v C}ech manner and take the homology of
the associated total complex.

\begin{lem}\label{Mgraded}
For any toric variety $X=X(\Delta)$,
the groups $HH_*(X)$ have a natural $M$-grading, and the endomorphism
$\theta_c$ maps the weight $m$ summand to the weight $cm$ summand.
\end{lem}

\begin{proof}
We have seen that the Hochschild complexes forming the columns of the
bicomplex are $M$-graded.
Since the ring maps are all $M$-graded, the {\v C}ech differentials
are also $M$-graded. Since $HH_*(X)$ is the homology of an $M$-graded
bicomplex, it is $M$-graded. Since the map $\theta_c$ sends the
weight $m$ subcomplex to the weight $cm$ subcomplex,
it has the same effect on homology.
\end{proof}

\begin{subremark}\label{CechSS}
This construction implies that the {\v C}ech spectral spectral sequence
is $M$-graded:
$$
E^1_{pq} = \bigoplus_{i_0< \dots< i_p}
HH_q(U_{\sigma_{i_0} \cap \cdots \cap \sigma_{i_p}})
\Rightarrow HH_{q-p}(X).
$$
\end{subremark}

\begin{lem}\label{invert-m}
Set $A=\sigma^\vee\cap M$.
If $m\in A$ lies on no proper face of $\sigma^\vee$, then
$A+\langle-m\rangle=M$, and $k[A][\chi^{-m}]=k[M]$.
\end{lem}

\begin{proof}
Since $k[A][\chi^{-m}]=k[A+\langle-m\rangle]$, it suffices to prove
the first assertion, \ie that every $t\in M$ is of the form $a-im$ for
some positive integer $i$. Fix a nonzero $n\in N$. The assumption that
$m$ lies on no proper face of $\sigma^\vee$ implies that
$\langle m,n\rangle>0$.  Hence $\langle t+im,n\rangle>0$ for $i\gg0$.
Since $\sigma\cap N$ is finitely generated,
it follows that $t+im\in A$ for $i\gg0$, as claimed.
\end{proof}

\begin{lem}\label{lem:scale}
The map $\theta_c:\Omega^q(U_\sigma)_m \to\Omega^q(U_\sigma)_{cm}$
is multiplication by $c^q\chi^{(c-1)m}$.
\end{lem}

\begin{proof} When $\sum u_i=m$, $\theta_c$ takes
$\omega=\chi^{u_0}d\chi^{u_1}\wedge\cdots d\chi^{u_q}$
to $c^q\chi^m\omega$.
\end{proof}

\begin{subremark}\label{t-scale}
The same proof shows that the map
$\theta_c:\tOmega^q(U_\sigma)_m\to\tOmega^q(U_\sigma)_{cm}$
is multiplication by $c^q\chi^{(c-1)m}$. By \eqref{E1},
this is an isomorphism for all $c\ne0$.
\end{subremark}

\begin{prop}\label{cOmega}
For any toric $k$-variety $X$, the natural maps \eqref{E3} induce
isomorphisms, for all $q$:
$$
\Omega^q(X)^\fc  \to \tOmega^q(X)^\fc
$$
\end{prop}

\begin{proof}
We may assume $X=U_\sigma$, so that $\Omega^q(X)=\Omega^q_{k[A]}$ for
$A=\sigma^\vee\cap M$.
It suffices to check that the map is an isomorphism in each
weight $m \in M_\fc$; without loss of generality, one may assume
$m \in M$. By Lemma \ref{HH-m},
$(\Omega^q_{k[A]})_m\cong (\Omega^q_{k[B]})_m$, where $B=A\cap\sigma(m)^\perp$.
By Lemma \ref{lem:scale}, $\theta_c$ coincides with multiplication by
$c^q \chi^{(c-1)m}$ both as a map
$(\Omega^q_{k[A]})_m \to (\Omega^q_{k[A]})_{cm}$ and as a map
$(\Omega^q_{k[B]})_m \to (\Omega^q_{k[B]})_{cm}$. Hence the group
$$
\Omega^q(X)^\fc_m = \varinjlim \left(
(\Omega^q_{k[A]})_m ~\map{\theta_{c_1}}~
(\Omega^q_{k[A]})_{c_1m} ~\map{\theta_{c_2}}~  \cdots \right)
$$
is the weight $m$ part of the localization of $\Omega^q_{k[B]}$ at
$\chi^m$, \ie of $\Omega^q(k[B][\chi^{-m}])$.
By construction, $m$ is not on any proper face of
$\sigma(m)^\vee\cap\sigma(m)^\perp$.
By Lemma \ref{invert-m},
$$
\Omega^q(k[B][\chi^{-m}])_m \cong
\Omega^q(k[B+\langle-m\rangle])_m = \Omega^q(k[T])_m,
\qquad T = M \cap \sigma(m)^\perp.
$$
Since $T$ is a free abelian group, $(\Omega^q_{k[T]})_m \cong \wedge^q(T)\otimes k$.
Now recall that by Remark \ref{t-scale} and (\ref{E1}) we also have
$$
(\tOmega^q_{k[T]})^{\fc}_m \cong
\tOmega^q(U_\sigma)^{\fc}_m = \tOmega^q(U_\sigma)_m \cong
k\cdot\chi^m\otimes \wedge^q(T),
$$
The map $(\Omega^q_{k[T)})_m\to (\tOmega^q_{k[T)})_m$ is given
by \eqref{E3}, and it is an isomorphism by inspection.
\end{proof}

In order to prove an analogous result for Hochschild homology, we
need to briefly review the decomposition of Hochschild homology into
summands given by the (higher) Andr\'e-Quillen homology groups.
For more details, we refer the reader to \cite[3.5]{Lo} or \cite[8.8]{WHomo}.

For a commutative $k$-algebra $R$, one forms a simplicial polynomial
$k$-algebra $R_\mathdot$ and a simplicial ring map $R_\mathdot \to R$
which is a homotopy equivalence on underlying simplicial sets.
The (higher) {\it cotangent complex} $\bL^{(q)}_{X/k}$ is defined to
be the simplicial
$R$-module $R\otimes_{R_\mathdot}\Omega^q_{R_\mathdot}$, and the
Andr\'e-Quillen homology groups of $R$ are defined to be
$D_p^{(q)}(R) = H_p(\bL^{(q)}_{X/k})$.
The $R$-modules $D_p^{(q)}(R)$ are independent up to isomorphism of
the choices made.

In general, there is a natural spectral sequence of $R$-modules
$$
D_p^{(q)}(R) \Longrightarrow HH_{p+q}(R)
$$
and a natural $R$-module isomorphism $D_0^{(q)}(R) \cong \Omega^q_{R/k}.$
Since we are assuming $\chr(k) = 0$, this spectral sequence
degenerates to give a natural decomposition of $R$-modules
$$
HH_n(R) \cong \bigoplus_{p+q = n} D_p^{(q)}(R) = \Omega^q_{R/k} \oplus
\bigoplus_{p+q = n, p > 0} D_p^{(q)}(R).
$$

Since the Andr\'e-Quillen homology groups are functorial for ring
maps, the endomorphisms $\theta_{c_i}$ preserve this decomposition.

\begin{lem}\label{AQscale}
Let $U_\sigma$ be an affine toric variety. Then the
$D_p^{(q)}(U_\sigma)$ are $M$-graded modules and,
for every $m\in\sigma^\vee\cap M$,
the map $\theta_c:D_p^{(q)}(U_\sigma)_m \to D_p^{(q)}(U_\sigma)_{cm}$
is multiplication by $c^q\chi^{(c-1)m}$.

\end{lem}

\begin{proof}
Let $A = \sigma^\vee \cap M$ and form a simplicial
resolution of $A$ by free abelian monoids $A_\mathdot \to A$. That is,
$A_\mathdot$ is a simplicial abelian monoid which in each degree is
free abelian and the map of simplicial abelian monoids $A_\mathdot \to A$
is a homotopy equivalence.
This is
possible by the same basic cotriple resolution used to form simplicial free
resolutions of $k$-algebras (see \cite[8.6]{WHomo}).
For functorial reasons, $k[A_\mathdot] \to k[A]$ is a free
simplicial resolution of $k[A]$.
We therefore have
$$
D_p^{(q)}(k[A]) = H_p(k[A]\otimes_{k[A_\mathdot]}\Omega^q_{k[A_\mathdot]}).
$$
For each $n$, the ring $k[A_n]$ is $M$-graded by the
maps $\delta_n: A_n \to A \subset M$. Thus the simplicial ring
$k[A_\mathdot]$ is also $M$-graded and the map $k[A_\mathdot] \to k[A]$ of
simplicial rings preserves this grading. It follows that
$k[A]\otimes_{k[A_\mathdot]}\Omega^q_{k[A_\mathdot]}$
is naturally $M$-graded, where the weight of
$\chi^{u_0}\otimes d(\chi^{u_1}) \wedge\cdots\wedge d(\chi^{u_q})$
is $u_0+\delta_n(u_1) + \cdots + \delta_n(u_q)$, for
any $u_0 \in A$ and $u_1, \dots, u_q \in A_n$.  Hence $D_p^{(q)}(k[A])$
is an $M$-graded $k[A]$-module, and it is clear that, for any positive
integer $c$, the endomorphism $\theta_c$ of $D_p^{(q)}(k[A])$ maps the
weight $m$ summand to the weight $cm$ summand.
To prove that the map
$$
\theta_c: D_p^{(q)}(k[A])_m \to D_p^{(q)}(k[A])_{cm}
$$
coincides with multiplication by $c^q \chi^{(c-1)m}$, it
suffices to prove the analogous assertion for the
$M$-graded $k[A]$-modules
$
k[A]\otimes_{k[A_n]}\Omega^q_{k[A_n]}.
$
The proof of this is exactly like the proof of Lemma \ref{lem:scale},
using $\omega=\chi^{u_0}\otimes d\chi^{u_1} \wedge\cdots\wedge d\chi^{u_q}$.
\end{proof}

\begin{thm} \label{Thm2}
For any toric $k$-variety $X$, the natural maps
$$
\Omega^q(X)^\fc \to HH_q(X) ^\fc
$$
are isomorphisms, for all $q$.
\end{thm}

\begin{proof}
By the spectral sequence in \ref{CechSS},
we may assume that $X$ is affine, say of the
form $X = U_\sigma$ for some cone $\sigma$. Setting $A=\sigma^\vee \cap M$,
the coordinate ring of $X$ is $k[A]$.
To establish the isomorphism $\Omega^p(U_\sigma)^\fc \cong
HH_p(U_\sigma)^\fc$ it suffices to prove that
$$
D_p^{(q)}(k[\sigma^\vee \cap M])^\fc = 0
$$
for all $p > 0$. As in the proof of Proposition \ref{cOmega},
it suffices to fix an arbitrary $m\in M$ and show that the
weight $m$ part vanishes. By Lemma \ref{HH-m},
$D_p^{(q)}(k[A])_m \cong D_p^{(q)}(k[B])_m$, where
$B=A\cap\sigma(m)^\perp$. By Lemma \ref{AQscale},
$\theta_c$ coincides with multiplication by $c^q \chi^{(c-1)m}$ both
as a map $D_p^{(q)}(k[A])_m\to D_p^{(q)}(k[A])_{cm}$ and as a map
$D_p^{(q)}(k[B])_m\to D_p^{(q)}(k[B])_{cm}$. Hence the weight $m$ summand
$$
D_p^{(q)}(X)^\fc_m = \varinjlim \left(
D_p^{(q)}(k[A])_m ~\map{\theta_{c_1}}~
D_p^{(q)}(k[A])_{c_1m} ~\map{\theta_{c_2}}~ \cdots \right)
$$
is the weight $m$ part of the localization of $D_p^{(q)}(k[B])$ at
$\chi^m$, \ie of $D_p^{(q)}(k[B][\chi^{-m}])$.

Recall that $\sigma(m) \subset \sigma$ denotes the face of $\sigma$
(possibly just the origin) on which $m = 0$. By Lemma \ref{invert-m},
$$
D_p^{(q)}(k[B][\chi^{-m}])_m \cong
D_p^{(q)}(k[B+\langle-m\rangle])_m = D_p^{(q)}(k[T])_m,
\quad T = M \cap \sigma(m)^\perp.
$$
Since $T=M\cap\sigma(m)^\perp$ is a free abelian group, we have
$$
D_p^{(q)}(k[B][\frac{1}{\chi^m}]) = D_p^{(q)}(k[T])=0
$$
for all $p>0$. This proves that $D_p^{(q)}(k[A])^\fc = 0$
for all $p>0$, proving the theorem.
\end{proof}

\begin{cor}\label{cor2}
For any field $k$ of characteristic $0$ and any toric
$k$-variety $X$, we have a natural isomorphism for all $n$:
$$
HH_n(X/\Q)^\fc \map{\simeq} \bH_\cdh^{-n}(X,HH(-/\Q))^{\fc}.
$$
\end{cor}
\noindent
The right hand side in \ref{cor2} denotes Hochschild homology with $cdh$
descent imposed (and localized by $\fc$). (On both sides, we take
Hochschild homology over $\Q$.)

\begin{proof} Let us write $X_\Q$ for the model of $X$ defined over
the rationals and $X_k = X$ for the model over $k$. We have
$X_k = X_\Q \times_{\Spec \Q} \Spec k$.

The natural map
$$
HH_n({X_k}_{/k})^\fc \map{} \bH^{-n}_\cdh(X,HH)^\fc
$$
is an isomorphism. Since
both sides satisfy Zariski descent, this is an
  immediate consequence Theorem \ref{Thm1} and Theorem \ref{Thm2}.
The K\"unneth formula for Hochschild homology, described before
Corollary \ref{HH/Q}, gives
$$
HH_*({X/\Q})^\fc \cong HH_*({X_\Q}/\Q)^\fc \otimes_\Q \Omega^*_{k/\Q}.
$$
In particular, one gets long exact sequences for
$HH_*({-/\Q})^\fc$
associated to abstract blow-ups of toric $k$-varieties.
Since the map
$$
HH_n(X_k/\Q)^\fc \cong \bH_\cdh^{-n}(X,HH(-/\Q))^{\fc}
$$
is an isomorphism whenever $X$ is smooth by
\cite[2.4]{CHW}, the result holds by induction and the five-lemma.
\end{proof}

\begin{cor} For any field $k$ of characteristic $0$ and any toric
$k$-variety $X$, and all $n$, we have
$$
HC_n(X/\Q)^\fc \cong \bH_\cdh^{-n}(X,HC(-/\Q))^{\fc}.
$$
\end{cor}

\begin{proof}
There is a map from the SBI sequence for $HH$ and $HC$ to the SBI sequence
for its $cdh$-fibrant variant. Applying the exact functor $(-)^\fc$
yields a similar map of long exact sequences, every third term of which
is the isomorphism of Corollary \ref{cor2}. The result now follows by
induction on $n$, since all complexes are cohomologically bounded above.
\end{proof}

\goodbreak
\begin{thm} \label{Thm3}
For any field $k$ of characteristic $0$ and any toric
 $k$-variety $X$, we have
$$
K_*(X)^\fc \cong KH_*(X)^\fc.
$$
\end{thm}

\begin{proof}
Since $(-)^{\fc}$ is exact, it suffices by Theorem \ref{FKFHC} to
show that $H^*_\zar(X,\cF_{HC})^\fc$ vanishes. Again because
$(-)^{\fc}$ is exact, we have a long exact sequence
$$
\cdots\to H^n_\zar(X,\cF_{HC})^\fc \to HC_{-n}(X/\Q)^\fc \to
\bH_\cdh^n(X,HC(-/\Q))^{\fc}
$$
The desired vanishing follows from the previous corollary.
\end{proof}

\begin{cor}\label{cor:DT} (Gubeladze's Dilation Theorem)
Let $\Gamma$ be an arbitrary commutative, cancellative, torsionfree monoid
without nontrivial units. Then for every sequence $\fc$ and every $p$,
$(K_p(k[\Gamma])/K_p(k))^\fc=0$.
\end{cor}

\begin{proof}
To prove the Dilation Theorem, it suffices to prove it
for all ``affine positive normal'' monoids, \ie for monoids of the form
$\Gamma=\sigma^\vee\cap M$ such that $\sigma^\perp=0$.
This is a reformulation of \cite[3.4]{Gu95},
and is stated explicitly in \cite[Proposition 2.1]{Gu05}
(up to the typo that $K_p(R[M])$ should be $K_p(R[M])/K_p(R)$).

For such $\Gamma$, $X=\Spec(k[\Gamma])$ is a toric variety, and
the proof of Proposition \ref{Willies} above shows that
$k[\Gamma]$ is $\N$-graded with $k$ in weight $0$. Hence
$KH(X) \simeq K(\Spec k)$. The result
now follows from Theorem \ref{Thm3}.
\end{proof}

\begin{rem} In  \cite{Gu07a}, Gubeladze proves an unstable version of
  his Dilation Theorem for the groups $K_1$ and $K_2$, which is valid for any
  regular coefficient ring in place of the field $k$.  In
  \cite{Gu07b}, he proves that his Dilation Theorem remains valid if
  one replaces the field $k$ by any
  regular coefficient ring that contains a copy of $\Q$.

\end{rem}

\medskip
\goodbreak

\goodbreak
\subsection*{Acknowledgements}
The third author thanks Joseph Gubeladze and Srikanth Iyengar for
useful conversations that contributed to this
paper.

\end{document}